\documentclass{amsart}

\usepackage{amssymb}
\usepackage{amsfonts}
\usepackage{amsmath}
\usepackage{amsthm}
\usepackage{amsrefs}
\usepackage{graphicx}
\usepackage{mathrsfs}
\usepackage{dsfont}
\usepackage[scaled]{helvet} 
\usepackage{courier} 
\normalfont
\usepackage[T1]{fontenc}
\usepackage{calligra}
\usepackage[usenames,dvipsnames]{color}
\usepackage[colorlinks=true,citecolor=Violet,linkcolor=blue]{hyperref}

\newcommand{\N}{{\mathds{N}}}

\newcommand{\R}{{\mathds{R}}}
\newcommand{\C}{{\mathds{C}}}

\newcommand{\D}{{\mathfrak{D}}}
\newcommand{\A}{{\mathfrak{A}}}
\newcommand{\B}{{\mathfrak{B}}}

\newcommand{\Lip}{{\mathsf{L}}}

\newcommand{\dpropinquity}[1]{{\mathsf{\Lambda}^{\ast}_{#1}}}
\newcommand{\propinquity}[1]{{\mathsf{\Lambda}^\prime_{#1}}}

\newcommand{\Kantorovich}[1]{{\mathsf{mk}_{#1}}}

\newcommand{\Haus}[1]{{\mathsf{Haus}_{#1}}}

\newcommand{\StateSpace}{{\mathscr{S}}}

\newcommand{\mongekant}{{Mon\-ge-Kan\-to\-ro\-vich metric}}

\newcommand{\Lqcms}{{\JLL} quantum compact metric space}

\newcommand{\unit}{1}

\newcommand{\sa}[1]{{\mathfrak{sa}\left({#1}\right)}}

\newcommand{\JLL}{Lei\-bniz}

\newcommand{\dom}[1]{{\operatorname*{dom}\left({#1}\right)}}
\newcommand{\codom}[1]{{\operatorname*{codom}\left({#1}\right)}}
\newcommand{\diam}[2]{{\mathrm{diam}\left({#1},{#2}\right)}}
\newcommand{\norm}[2]{{\left\|{#1}\right\|_{#2}}}
\newcommand{\tunnelset}[3]{{\text{\calligra Tunnels}\,\left[\left({#1}\right)\stackrel{#3}{\rightarrow}\left({#2}\right)\right]}}
\newcommand{\journeyset}[3]{{\text{\calligra Journeys}\left[\left({#2}\right)\stackrel{#1}{\longrightarrow}\left({#3}\right)\right]}}

\newcommand{\Jordan}[2]{{{#1}\circ{#2}}} 
\newcommand{\Lie}[2]{{\left\{{#1},{#2}\right\}}} 

\newcommand{\tunneldepth}[2]{{\delta\left( {#1} \right)}}
\newcommand{\tunnelreach}[2]{{\rho\left( {#1} \right)}}
\newcommand{\tunnellength}[2]{{\lambda\left({#1} \right)}}

\newcommand{\tunnelextent}[2]{{\chi\left({#1}\right)}}

\newcommand{\journeylength}[2]{{\lambda\left(#1\right)}}

\newcommand{\co}[1]{{\overline{\mathrm{co}}\left(#1\right)}}
\newcommand{\alg}[1]{{\mathfrak{#1}}}

\newcommand{\Set}[2]{{\left\{\begin{aligned} #1 \end{aligned} \middle\vert\begin{aligned} #2 \end{aligned}\right\}}}

\theoremstyle{plain}
\newtheorem{theorem}{Theorem}[section]

\newtheorem{proposition}[theorem]{Proposition}

\newtheorem{theorem-definition}[theorem]{Theorem-Definition}
\newtheorem{proposition-definition}[theorem]{Proposition-Definition}

\theoremstyle{definition}
\newtheorem{definition}[theorem]{Definition}

\newtheorem{notation}[theorem]{Notation}

\theoremstyle{remark}

\renewcommand{\leq}{\leqslant}

\numberwithin{equation}{section}

\begin{document}

\title{The Triangle Inequality and the Dual Gro\-mov-\-Haus\-dorff Propinquity}
\author{Fr\'{e}d\'{e}ric Latr\'{e}moli\`{e}re}
\email{frederic@math.du.edu}
\urladdr{http://www.math.du.edu/\symbol{126}frederic}
\address{Department of Mathematics \\ University of Denver \\ Denver CO 80208}

\date{\today}

\begin{abstract}
The dual Gro\-mov-\-Haus\-dorff propinquity is a generalization of the Gro\-mov-\-Haus\-dorff distance to the class of {\Lqcms s}, designed to be well-behaved with respect to C*-algebraic structures. In this paper, we present a variant of the dual propinquity for which the triangle inequality is established without the recourse to the notion of journeys, or finite paths of tunnels. Since the triangle inequality has been a challenge to establish within the setting of {\Lqcms s} for quite some time, and since journeys can be a complicated tool, this new form of the dual propinquity is a significant theoretical and practical improvement. On the other hand, our new metric is equivalent to the dual propinquity, and thus inherits all its properties.
\end{abstract}

\maketitle


\section{Introduction}

The quest for a generalization of the Gromov-Hausdorff distance \cites{Gromov81,Gromov} to the realm of Connes' noncommutative geometry \cite{Connes}, initiated with Rieffel's introduction of the quantum Gromov-Hausdorff distance \cite{Rieffel00}, soon raised the challenge of constructing an appropriate distance on the class of compact quantum metric spaces described by C*-algebras equipped with Leibniz seminorms. Several desirable features which such a distance should possess emerged from recent research in noncommutative metric geometry \cites{Rieffel09,Rieffel10,Rieffel10c,Rieffel11,Rieffel12}, and the search for such a metric proved to be a delicate matter. Yet, the apparent potential of extending the methods of metric geometry \cites{Gromov81,Gromov} to noncommutative geometry, as well as the potential applications to mathematical physics, are strong motivations to try and meet this challenge.

A surprising difficulty, among many, encountered in the construction of an appropriate generalization of the Gromov-Hausdorff distance to the class of {\Lqcms s}, was the proof of the triangle inequality for such prospective distances. Both the need and challenge raised by this matter are expressed in \cite{Rieffel10c}, following on earlier work \cites{Rieffel09,Rieffel10}, where the construction of a variation on the original quantum Gromov-Hausdorff distance \cite{Rieffel00}, called the quantum proximity, may or not satisfy the triangle inequality.

The Dual Gromov-Hausdorff Propinquity \cite{Latremoliere13c} is a distance introduced by the author as a possible answer to the challenges raised by working within the category of {\Lqcms s}. The purpose of the present article is to prove that, in fact, an equivalent variation of the dual propinquity satisfies the triangle inequality without recourse to some of the less natural aspects of the construction in \cite{Latremoliere13c}, and thus provides a satisfactory answer to a long standing challenge. The proof of the triangle inequality which we present strengthen our understanding of our dual propinquity, provides a much more effective tool for both the theoretical study of our metric and for its applications, and settle in a more satisfactory manner than our previous work did \cite{Latremoliere13} the vexing question of establishing the triangle inequality for noncommutative Gromov-Hausdorff distances adapted to C*-algebraic structures.

Inspired by the groundbreaking work of Connes \cites{Connes97,Connes} in noncommutative geometry and its metric aspects, Rieffel introduced the notion of a compact quantum metric space \cites{Rieffel98a,Rieffel99,Rieffel02,li03,Ozawa05} and the quantum Gromov-Hausdorff distance \cites{Rieffel00,Rieffel01,Rieffel10,Rieffel10c,Rieffel11}, a fascinating generalization of Gromov-Hausdorff distance to noncommutative geometry. While various algebraic structures can be considered as a foundation for a theory of quantum compact metric spaces, from order-unit spaces in the original work of Rieffel, to operator systems \cite{kerr02}, operator spaces \cite{Wu06a}, and others, recent research in noncommutative metric geometry \cites{Rieffel09,Rieffel10,Rieffel10c,Rieffel11,Rieffel12} suggests that a desirable framework consists of the class of {\Lqcms s} \cite{Latremoliere13}, or some of its subclasses such as compact C*-metric spaces \cite{Rieffel10c}. 

However, the quantum Gromov-Hausdorff distance may be zero between two {\Lqcms s} whose underlying C*-algebras are not *-isomorphic, and allows the use of non-Leibniz seminorms for its computation. As seen for instance in \cites{Rieffel09,Rieffel10,Rieffel10c}, the study of the behavior of C*-algebraic related structures such as projective modules under metric convergence is facilitated if the notion of metric convergence is aware of the C*-algebraic structure, which would in particular translate in distance zero implying isomorphism of the underlying C*-algebras, and would in fact allow to work entirely within the class of {\Lqcms s}.

Much effort was spent to find various strengthening of the quantum Gromov-Hausdorff distance with the desirable coincidence property for C*-algebras. The first successful construction of such a metric was done by Kerr \cite{kerr02}, and it was soon followed with other metrics \cites{li03,Kerr09} designed to address the same issue. Yet, the Leibniz property of seminorms played no role in the construction of these metrics, which instead relied on replacing, in the original construction of the quantum Gromov-Hausdorff distance, the state space of the compact quantum metric spaces and its natural {\mongekant} by more complex structures, from the family of matrix valued completely positive unital linear maps to the restriction of the graph of the multiplication to the noncommutative analogue of the Lipschitz seminorm unit ball. Other approaches to noncommutative metric geometry include the quantized Gromov-Hausdorff distance \cite{Wu06b} adapted to operator spaces. Thus, all these new metrics were adapted to various categories of compact quantum metric spaces and had their own fascinating features, yet the question of dealing with C*-algebras and Leibniz seminorms was not fully settled.

Motivated by his work on the relationship between vector bundles and the Gromov-Hausdorff distance \cite{Rieffel09}, Rieffel introduced another strengthening of his original quantum Gromov-Hausdorff distance, the quantum proximity \cite{Rieffel10c}, in order to encode more C*-algebraic structure within his notion of metric convergence --- a problem deeply connected to the coincidence property issue. However, there is no known proof that the quantum proximity satisfies a form of the triangle inequality and it remains unclear whether it actually does, though it seems unlikely. The root of this difficulty is that the quotient of a Leibniz seminorm is not always Leibniz \cite{Blackadar91}, thus preventing the application of the methods used earlier by Rieffel \cite{Rieffel00} to establish the triangle inequality for the quantum Gromov-Hausdorff distance.

We begun our research on the topic of strengthening the quantum Gromov-Hausdorff distance motivated in part by our own research on the theory of quantum locally compact metric spaces \cites{Latremoliere05b,Latremoliere12b}. Our first step in our own approach to this problem was to understand what seems to be a crucial construction in noncommutative metric geometry where metrics were constructed using bi-module C*-algebras \cites{Rieffel10c, Rieffel11}. We proposed a metric which was adapted to this situation, called the quantum Gromov-Hausdorff propinquity \cite{Latremoliere13}, which answered several of the challenges mentioned in the literature on the quantum Gromov-Hausdorff distance: it explains how to use the bimodule construction to get estimates on the quantum Gromov-Hausdorff distance (as our new metric dominates Rieffel's distance), and allowed one to work entirely within the category of {\Lqcms s}, which we introduced in that paper. Our quantum propinquity also enjoyed the desired coincidence property where distance zero implies *-isomorphism. In particular, the quantum propinquity satisfied the triangle inequality, and seemed a good replacement for the quantum proximity from which it takes its name. Some of the main examples of convergences for the quantum Gromov-Hausdorff distance, such as finite dimensional approximations of quantum tori \cite{Latremoliere05} and finite dimensional approximations of spheres \cite{Rieffel01}, can be strengthened to hold for the quantum propinquity \cites{Rieffel10c,Latremoliere13b}.

We soon noticed that the favorable algebraic framework provided by bimodules in the study of {\Lqcms s}, as used for the quantum propinquity, was also a source of some rigidity which made some questions, such as completeness, difficult to tackle. We thus introduced a new metric, the dual Gromov-Hausdorff propinquity \cite{Latremoliere13c}, originally to study the completeness question of the quantum propinquity. Yet the dual propinquity also met the challenge of providing a generalization of the Gromov-Hausdorff distance to {\Lqcms s}, having the desirable coincidence property, the triangle inequality, only involving {\Lqcms s} in its construction, offering the flexibility needed to work, if so wished, within proper subclasses of {\Lqcms s} such as compact C*-metric spaces \cite{Rieffel10c}, with the additional and essential property of being a complete metric. Thus, the dual propinquity defines a natural framework for the study of Gromov-Hausdorff convergence of {\Lqcms s}. As it is dominated by the quantum propinquity, it inherits all of its known examples of convergence as well. Moreover, completeness provides additional methods to prove some of these convergences, such as in the case of finite dimensional approximations of spheres \cite{Rieffel11}.

Yet, the proof of the triangle inequality for the dual propinquity relied on the same technique we employed for the quantum propinquity, and involves, essentially, the use of finite paths of {\Lqcms s}. While effective, and still the only method we know how to use for the quantum propinquity, this approach is the only rather artificial aspect of the dual propinquity and its construction. It also adds a layer of complexity in the studies of the properties of the dual propinquity. A more natural proof of the triangle inequality is a very desirable feature, in view of the brief historical review of the evolution of the field over the past decade and a half: as a major obstacle which took some time to be overcome, we believe it is very worthwhile to get the most efficient proof of this core feature of the dual propinquity. It would also solidify our claim that the dual propinquity offers a natural framework for the study of {\Lqcms s}. 

This paper proposes a natural proof of the triangle inequality for the dual propinquity by exploiting a particular flexibility of our construction in \cite{Latremoliere13c}, in exchange for a mild variation in the definition of the dual propinquity which leads to an equivalent metric. The reason why this proof has remained hidden for some time lies in the following observation: the Gromov-Hausdorff distance between two compact metric spaces $(X,\mathsf{d}_X)$ and $(Y,\mathsf{d}_Y)$ may be defined by identifying these spaces with their copies in their disjoint union $X\coprod Y$, and then take the infimum of their Hausdorff distance for all possible metrics on $X\coprod Y$ which restricts to $\mathds{d}_X$ and $\mathsf{d}_Y$ on $X$ and $Y$, respectively (such metrics on $X\coprod Y$ are called admissible). Thus, Rieffel's original construction of the quantum Gromov-Hausdorff distance between two compact quantum metric spaces $\mathds{A}$ and $\mathds{B}$ relies on identifying these two spaces with their copies in their direct sum $\mathds{A}\oplus\mathds{B}$. This feature of Rieffel's construction has been repeated in all the constructions of noncommutative Gromov-Hausdorff distances, explicitly or implicitly (for instance, it is implicit in our construction of the quantum propinquity, as becomes visible in \cite[Theorem 6.3]{Latremoliere13}).

Yet, the dual propinquity does not employ the same scheme. The Gromov-Hausdorff distance between two compact metric spaces $(X,\mathsf{d}_X)$ and $(Y,\mathsf{d}_Y)$ may also be computed as the infimum of the Hausdorff distance between $\iota_X(X)$ and $\iota_Y(Y)$ where $\iota_X : X\hookrightarrow X$ and $\iota_Y: T\hookrightarrow Z$ are two isometries into a compact metric space $(Z,\mathsf{d}_Z)$. The dual propinquity follows this scheme, employing in its construction arbitrary isometric embeddings in {\Lqcms s}, though it does measure how far from being a direct sum such embeddings are. This latter measurement is necessary for the properties of the dual propinquity to hold, such as completeness.

In the classical picture, the two constructions of the Gromov-Hausdorff distance given in the two past paragraphs are easily seen to be equivalent. In particular, one can obtain an admissible metric on $X\coprod Y$ from any pair of isometric embeddings $\iota_X$ and $\iota_Y$, using the notations in the previous few paragraphs, and to do so involve at some point restricting a metric to a subspace. This process, in the noncommutative picture, becomes taking a quotient, and as we have indicated, Leibniz seminorms are not well-behaved with respect to quotient. Thus, this must be avoided, and to this end, we defined the dual propinquity in the more general context described above.

Remarkably, the proof of the triangle inequality which we now offer takes full advantage of this very remark. As we shall see, the proof does not carry to Rieffel's proximity, for instance, precisely because of this key difference in the basic construct of the quantum proximity and the dual propinquity. This observation strongly suggests that the approach of the dual propinquity is quite appropriate when working with {\Lqcms s}, and in particular the added flexibility it affords us allows to settle the triangle inequality in a very satisfactory manner. The only slight matter, which we address in this article, is that the two basic quantities involved in the definition of the dual propinquity should be combined in a somewhat more subtle manner than was done in \cite{Latremoliere13c} to allow the proof of the triangle inequality (otherwise, a factor of $2$ intervenes in a somewhat undesirable, if not dramatic, fashion). There is little concern, however, in doing this change, as we obtain an equivalent metric.

We will avoid introducing excessive terminology when introducing the variant of the dual propinquity in this paper; we shall henceforth refer to this variant simply as the Gromov-Hausdorff propinquity, or even propinquity, and we employ very similar notations for the new propinquity and the dual propinquity, as we really see them as the same metric for all intent and purposes. This paper complement \cite{Latremoliere13c}, by adding to our understanding of the dual propinquity.

We begin our exposition with the notion of the extent of a tunnel, which is a quantity which combines the depth and reach of tunnels \cite{Latremoliere13c} in a convenient manner. We then show how this new quantity enjoy a form of triangle inequality, obtained by composing tunnels, rather than journeys, under appropriate assumptions. We then conclude by introducing our variant of the dual propinquity and establish its core properties.

\section{Extent of Tunnels}

Noncommutative metric geometry proposes to study noncommutative generalizations of algebras of Lipschitz functions on metric spaces. For our purpose, the basic class of objects is given in the following definition, first proposed in \cite{Latremoliere13}, as a C*-algebraic variation of Rieffel's original notion of compact quantum metric space \cites{Rieffel98a,Rieffel99}. We also use the following definition to set all our basic notations in this paper.

\begin{definition}\label{Lcqms-def}
A {\Lqcms} $(\A,\Lip)$ is a pair of a unital C*-algebra $\A$ with unit $\unit_\A$ and a seminorm $\Lip$ defined on a dense Jordan-Lie subalgebra $\dom{\Lip}$ of the self-adjoint part $\sa{\A}$ of $\A$, such that:
\begin{enumerate}
\item $\{a\in\dom{\Lip} : \Lip(a) = 0 \} = \R\unit_\A$,
\item the {\mongekant} $\Kantorovich{\Lip}$, defined on the state space $\StateSpace(\A)$ of $\A$ by:
\begin{equation*}
  \Kantorovich{\Lip} : \varphi,\psi \in \StateSpace(\A) \longmapsto \sup\left\{ |\varphi(a) - \psi(a)| : a\in\dom{\Lip}, \Lip(a) \leq 1 \right\}
\end{equation*}
metrizes the weak* topology on $\StateSpace(\A)$,
\item for all $a,b \in \dom{\Lip}$ we have:
\begin{equation}\label{Leibniz-Jordan-eq}
\Lip\left(\Jordan{a}{b}\right) \leq \|a\|_\A\Lip(b) + \|b\|_\A\Lip(a)
\end{equation}
and:
\begin{equation}\label{Leibniz-Lie-eq}
\Lip\left(\Lie{a}{b}\right) \leq \|a\|_\A\Lip(b) + \|b\|_\A\Lip(a)\text{,}
\end{equation}
where $\Jordan{a}{b} = \frac{ab+ba}{2}$ and $\Lie{a}{b} = \frac{ab-ba}{2i}$,
\item the seminorm $\Lip$ is lower semi-continuous with respect to the norm $\|\cdot\|_\A$ of $\A$.
\end{enumerate}
When $(\A,\Lip)$ is a {\Lqcms}, the seminorm $\Lip$ is called a Leibniz Lip-norm on $\A$.
\end{definition}

As a usual convention, given a seminorm $\Lip$ defined on some dense subspace $\dom{\Lip}$ of a space $\A$, we set $\Lip(a) = \infty$ whenever $a\not\in\dom{\Lip}$.

The {\Lqcms s} are objects of a natural category (and in fact, \cite{Latremoliere13c} introduces two categories with {\Lqcms s} as objects). For our purpose, we shall only use the following notion:
\begin{definition}
Let $(\A,\Lip_\A)$ and $(\B,\Lip_\B)$ be two {\Lqcms s}. An \emph{isometric isomorphism} $h : (\A,\Lip_\A) \rightarrow (\B,\Lip_\B)$ is a *-isomorphism from $\A$ onto $\B$ such that $\Lip_\B\circ h = \Lip_\A$.
\end{definition}

To construct a distance between {\Lqcms s}, we propose to extend the notion of an isometric embedding of two compact metric spaces into a third compact metric space to the noncommutative world. If $\iota_X : X\hookrightarrow Z$ is an isometry from a compact metric space $(X,\mathsf{d}_X)$ into a compact metric space $(Z,\mathsf{d}_Z)$, and if $f : X\rightarrow\R$ is a Lipschitz function, then McShane's Theorem \cite{McShane34} provides us with a function $g : Z\rightarrow \R$ such that $g\circ\iota_X = f$ and with the same Lipschitz constant on $(Z,\mathsf{d}_Z)$ as $f$ on $(X,\mathsf{d}_X)$. Thus the quotient of the Lipschitz seminorm $\Lip_Z$ associated with $\mathsf{d}_Z$ for $\iota_X$ is given by the Lipschitz seminorm $\Lip_X$ associated with $\mathsf{d}_X$. Conversely, if $\iota_X : X\hookrightarrow Z$ is an injection and if $\Lip_X$ is the quotient of $\Lip_Z$ for $\iota_X$ then $\iota_X$ is an isometry. This observation justifies our restriction to seminorms defined on the self-adjoint part of C*-algebras in Definition (\ref{Lcqms-def}) since extensions of $\C$-valued Lipschitz functions from subspace to the full space are not guaranteed to have the same Lipschitz seminorm \cite{Rieffel06}.

The duality between the notion of isometric embedding and quotient Lip-norms is the basis of Rieffel's construction of the quantum Gromov-Hausdorff distance \cite{Rieffel00}, though it is used only in a restricted form. In \cite{Latremoliere13c}, we proposed the following formulation of more general isometric embeddings for our purpose:

\begin{definition}\label{tunnel-def}
Let $(\A_1,\Lip_1)$ and $(\A_2,\Lip_2)$ be two {\Lqcms s}. A \emph{tunnel} $(\D,\Lip_\D,\pi_1,\pi_2)$ is a {\Lqcms} $(\D,\Lip_\D)$ and two *-epimorphisms $\pi_1 : \D\twoheadrightarrow\A_1$ and $\pi_2:\D\twoheadrightarrow\A_2$ such that, for all $j\in\{1,2\}$, and for all $a\in \sa{\A_j}$, we have:
\begin{equation*}
\Lip_j(a) = \inf\left\{ \Lip_\D(d) : d\in\sa{\D}, \pi_j(d) = a \right\}\text{.}
\end{equation*}
\end{definition}

A tunnel provides a mean to measure how far apart two {\Lqcms s} might be. To this end, we introduced several quantities in \cite{Latremoliere13c} associated with tunnels. As a matter of notation, the Hausdorff distance on closed subsets of a compact metric space $(X,\mathsf{d}_X)$ is denoted by $\Haus{\mathsf{d}_X}$.

\begin{definition}\label{tunnel-length-def}
Let $(\A,\Lip_\A)$ and $(\B,\Lip_\B)$ be two {\Lqcms s}, and let $\tau = (\D,\Lip_\D,\pi_\A,\pi_\B)$ be a tunnel from $(\A,\Lip_\A)$ to $(\B,\Lip_\B)$. 
\begin{itemize}
\item The \emph{reach} $\tunnelreach{\tau}{}$ of $\tau$ is:
\begin{equation*}
\Haus{\Kantorovich{\Lip_\D}}(\pi_\A^\ast(\StateSpace(\A)),\pi_\B^\ast(\StateSpace(\B)))\text{,}
\end{equation*}
where for any positive linear map $f : \alg{E}\rightarrow\alg{F}$ between two C*-algebras $\alg{E}$ and $\alg{F}$, we denote the transpose map by $f^\ast : \varphi \in\StateSpace(\alg{F}) \mapsto \varphi\circ f \in \StateSpace(\alg{E})$.
\item The \emph{depth} $\tunneldepth{\tau}{}$ of $\tau$ is:
\begin{equation*}
\Haus{\Kantorovich{\Lip_\D}}(\StateSpace(\D),\co{\StateSpace(\A)\cup\StateSpace(\B)})\text{,}
\end{equation*}
where $\co{E}$ is the closure of the convex envelope of a set $E \subseteq \StateSpace(\D)$.
\item The \emph{length} $\tunnellength{\tau}{}$ of $\tau$ is:
\begin{equation*}
\max\left\{ \tunneldepth{\tau}{}, \tunnelreach{\tau}{} \right\}\text{.}
\end{equation*}
\end{itemize}
\end{definition}

A feature of the dual propinquity is that it allows to put restrictions on which tunnels to use to compute the distance between two {\Lqcms s}. For instance, using the notations of Definition (\ref{tunnel-def}), one could only consider tunnels of the form $(\A\oplus\B,\Lip,\pi_\A,\pi_\B)$ where $\pi_\A$ and $\pi_\B$ are the canonical surjections and $\Lip$ is admissible for $\Lip_\A$ and $\Lip_\B$ in the sense of \cite{Rieffel00}. A more interesting restriction could be to ask for tunnels $(\D,\Lip,\pi,\rho)$ with $(\D,\Lip)$ a compact C*-metric space \cite{Rieffel10c}. Many other variations may occur depending on the desired application, as long as we choose a class of tunnels with the following properties needed to ensure that the dual propinquity is indeed a metric, as described in the next three definitions:

\begin{definition}[\cite{Latremoliere13c}*{Definition 3.10}]
Let $(\A,\Lip_\A)$ and $(\B,\Lip_\B)$ be two {\Lqcms s} and let $\tau = (\D,\Lip_\D,\pi_\A,\pi_\B)$ be a tunnel from $(\A,\Lip_\A)$ to $(\B,\Lip_\B)$. The \emph{reversed tunnel} $\tau^{-1}$ of $\tau$ is the tunnel $(\D,\Lip_\D,\pi_\B,\pi_\A)$ from $(\B,\Lip_\B)$ to $(\A,\Lip_\A)$.
\end{definition}

\begin{definition}[\cite{Latremoliere13c}*{Definition 3.18}]
Let $\mathcal{T}$ be a class of tunnels. Let $(\A,\Lip_\A)$ and $(\B,\Lip_\B)$ be two {\Lqcms s}. A \emph{journey} along $\mathcal{T}$ from $(\A,\Lip_\A)$ to $(\B,\Lip_\B)$ is a finite family:
\begin{equation*}
\left( \A_j,\Lip_j,\tau_j,\A_{j+1},\Lip_{j+1} : j=1,\ldots,n \right)
\end{equation*}
where $n\in\N$, and:
\begin{enumerate}
\item $(\A_1,\Lip_1) = (\A,\Lip_\A)$ and $(\A_{n+1},\Lip_{n+1}) = (\B,\Lip_\B)$,
\item for all $j\in\{1,\ldots,n+1\}$, the pair $(\A_j,\Lip_j)$ is a {\Lqcms},
\item for all $j\in\{1,\ldots,n\}$, the tunnel $\tau_j \in \mathcal{T}$ goes from $(\A_j,\Lip_j)$ to $(\A_{j+1},\Lip_{j+1})$.
\end{enumerate}
\end{definition}

\begin{notation}
Let $\mathcal{T}$ be a class of tunnels. Let $(\A,\Lip_\A)$ and $(\B,\Lip_\B)$ be two {\Lqcms s}. The set of all journeys from $(\A,\Lip_\A)$ to $(\B,\Lip_\B)$ along $\mathcal{T}$ is denoted by:
\begin{equation*}
\journeyset{\mathcal{T}}{\A,\Lip_\A}{\B,\Lip_\B} \text{.}
\end{equation*}
\end{notation}

\begin{definition}[\cite{Latremoliere13c}*{Definition 3.11}]
Let $\mathcal{C}$ be a nonempty class of {\Lqcms s}. A class $\mathcal{T}$ of tunnels is \emph{compatible} with $\mathcal{C}$ when:
\begin{enumerate}
\item if $\tau \in \mathcal{T}$ then $\tau^{-1}\in\mathcal{T}$,
\item if $(\A,\Lip_\A)$ and $(\B,\Lip_\B)$ are in $\mathcal{C}$ then $\journeyset{\mathcal{T}}{\A,\Lip_\A}{\B,\Lip_\B}$ is not empty,
\item if $(\A,\Lip_\A)$ and $(\B,\Lip_\B)$ are in $\mathcal{C}$ and $h : (\A,\Lip_\A)\rightarrow (\B,\Lip_\B)$ is an isometric isomorphism, then $(\A,\Lip_\A,\mathrm{id}_\A,h)$ and $(\B,\Lip_\B,h^{-1},\mathrm{id}_\B)$ are in $\mathcal{T}$, where $\mathrm{id}_\A$ and $\mathrm{id}_\B$ are the identity automorphisms of $\A$ and $\B$ respectively,
\item any tunnel in $\mathcal{T}$ is from and to an element of $\mathcal{C}$.
\end{enumerate}
\end{definition}

\begin{notation}
If $\mathcal{C}$ is a nonempty class of {\Lqcms s} and $\mathcal{T}$ is a class of tunnels compatible with $\mathcal{C}$, and if $(\A,\Lip_\A)$ and $(\B,\Lip_\B)$ are in $\mathcal{C}$, then the (nonempty) set of all tunnels in $\mathcal{T}$ from $(\A,\Lip_\A)$ to $(\B,\Lip_\B)$ is denoted by:
\begin{equation*}
\tunnelset{\A,\Lip_\A}{\B,\Lip_\B}{\mathcal{T}} \text{.}
\end{equation*}
\end{notation}

While a natural approach, taking the infimum of the lengths of all tunnels between two given {\Lqcms s} for some class of compatible tunnels does not a priori give us a metric. The largest pseudo-metric dominated by the lengths of tunnels is constructed in \cite{Latremoliere13c} and is called the dual Gro\-mov-\-Haus\-dorff propinquity:

\begin{definition}[\cite{Latremoliere13c}*{Definition 3.21}]
Let $\mathcal{C}$ be a nonempty class of {\Lqcms s} and let $\mathcal{T}$ be a class of tunnels compatible with $\mathcal{C}$. The \emph{dual Gromov-Hausdorff $\mathcal{T}$-propinquity} between two {\Lqcms s} $(\A,\Lip_\A)$ and $(\B,\Lip_\B)$ in $\mathcal{C}$ is:
\begin{equation*}
\dpropinquity{\mathcal{T}}((\A,\Lip_\A),(\B,\Lip_\B)) = \inf\left\{ \journeylength{\Upsilon}{} : \Upsilon \in \journeyset{\mathcal{T}}{\A,\Lip_\A}{\B,\Lip_\B} \right\}
\end{equation*}
where the length $\journeylength{\Upsilon}{}$ of a journey:
\begin{equation*}
\Upsilon = \left( \A_j,\Lip_j,\tau_j,\A_{j+1},\Lip_{j+1} : j=1,\ldots,n \right)
\end{equation*}
is $\sum_{j=1}^n\tunnellength{\tau_j}{}$.
\end{definition}

Remarkably, the dual Gromov-Hausdorff propinquity is a metric \cite[Theorem 4.17]{Latremoliere13c} up to isometric isomorphism, and it is complete \cite[Theorem 6.27]{Latremoliere13c} when applied to the class of all tunnels between arbitrary {\Lqcms s}.

We propose to modify the dual propinquity to avoid the recourse to journeys in its definition. To this end, we propose a new quantity associated with tunnels:

\begin{definition}\label{tunnel-extent-def}
Let $(\A,\Lip_\A)$ and $(\B,\Lip_\B)$ be two {\Lqcms s} and let $\tau = (\D,\Lip_\D,\pi_\A,\pi_\B)$ be a tunnel from $(\A,\Lip_\A)$ to $(\B,\Lip_\B)$. The \emph{extent} $\tunnelextent{\tau}{}$ of $\tau$ is the nonnegative real number:
\begin{equation*}
\max\left\{ \Haus{\Kantorovich{\Lip_\D}}(\StateSpace(\D),\StateSpace(\A)), \Haus{\Kantorovich{\Lip_\D}}(\StateSpace(\D),\StateSpace(\B)) \right\}\text{.}
\end{equation*}
\end{definition}

The key observation for this section is that the extent is equivalent to the length:

\begin{proposition}\label{length-vs-extent-prop}
Let $(\A,\Lip_\A)$ and $(\B,\Lip_\B)$ be two {\Lqcms s} and let $(\D,\Lip_\D,\pi_\A,\pi_\B)$ be a tunnel from $(\A,\Lip_\A)$ to $(\B,\Lip_\B)$. Then:
\begin{equation*}
\tunnellength{\tau}{} = \max\{ \tunnelreach{\tau}{\cdot}, \tunneldepth{\tau}{} \} \leq \tunnelextent{\tau}{\cdot} \leq 2 \tunnellength{\tau}{} \text{.}
\end{equation*}
\end{proposition}

\begin{proof}
If $\varphi \in \StateSpace(\A)$ then $\pi_\A^\ast(\varphi) \in\StateSpace(\D)$ and thus there exists $\psi \in \StateSpace(\B)$ such that:
\begin{equation*}
\Kantorovich{\Lip_\D}(\varphi\circ\pi_\A,\psi\circ\pi_\B) \leq \tunnelextent{\tau}{\cdot}\text{.}
\end{equation*}
As the argument is symmetric in $\A$ and $\B$, we deduce:
\begin{equation*}
\tunnelreach{\tau}{} \leq \tunnelextent{\tau}{\cdot}\text{.}
\end{equation*}

Now, by definition, since $\StateSpace(\A),\StateSpace(\B) \subseteq \co{\StateSpace(\A)\cup\StateSpace(\B)}$, we certainly have $\tunneldepth{\tau}{}\leq\tunnelextent{\tau}{}$. Hence:
\begin{equation*}
\max\{ \tunnelreach{\tau}{}, \tunneldepth{\tau}{} \}  \leq \tunnelextent{\tau}{\cdot}\text{.}
\end{equation*}

This proves the first inequality of our proposition. We now prove the second inequality.

Let $\varphi\in\StateSpace(\D)$. By Definition of the depth of $\tau$, There exists $\mu\in\StateSpace(\A)$ and $\nu\in\StateSpace(\B)$ such that, for some $t\in [0,1]$:
\begin{equation*}
\Kantorovich{\Lip_\D}(\varphi,t\mu\circ\pi_\A+(1-t)\nu\circ\pi_\B) \leq \tunneldepth{\tau}{}\text{.}
\end{equation*}

Now, by the definition of the reach of $\tau$, there exists $\theta\in \StateSpace(\A)$ such that:
\begin{equation*}
\Kantorovich{\Lip_\D}(\nu\circ\pi_\B,\theta\circ\pi_\A)\leq\tunnelreach{\tau}{}\text{.}
\end{equation*}
Let $\eta = t\mu\circ\pi_\A+(1-t)\nu\circ\pi_\B$ and $\omega = t\mu\circ\pi_\A+(1-t)\theta\circ\pi_\A$. Then:
\begin{equation*}
\begin{split}
\Kantorovich{\Lip_\D}(\varphi,\omega) &\leq \Kantorovich{\Lip_\D}(\varphi,\eta)  + \Kantorovich{\Lip_\D}(\eta,\omega)\\
&\leq \tunneldepth{\tau}{} + \sup\Set{|\eta(d) - \omega(d)| } {d\in\sa{\D},\Lip_\D(d) \leq 1 }\\
&= \tunneldepth{\tau}{} + \sup\Set{(1-t)|\nu(\pi_\B(d)) - \theta(\pi_\A(d))|}{d\in\sa{\D}, \Lip_\D(d) \leq 1}\\
&= \tunneldepth{\tau}{} + (1-t)\Kantorovich{\Lip_\D}(\nu\circ\pi_\B,\theta\circ\pi_\A)\\
&\leq \tunneldepth{\tau}{} + \tunnelreach{\tau}{}\text{.}
\end{split}
\end{equation*}
Again by symmetry in $\A$ and $\B$, we get:
\begin{equation*}
\tunnelextent{\tau}{\cdot} \leq \tunnelreach{\tau}{} + \tunneldepth{\tau}{} \leq 2\tunnellength{\tau}{} \text{.}
\end{equation*}
This concludes our proof.
\end{proof}

In \cite{Latremoliere13c}, tunnels were not composed; instead journeys were introduced because of their natural composition properties, from which the triangle inequality of the dual propinquity was proven. We now show that we can actually compose tunnels in such a manner that the extent behaves properly with respect to this composition. This leads us to our variant of the propinquity in the next section.

\section{Triangle Inequality for an equivalent metric to the Dual Propinquity}

The main result of this section is the following method to compose tunnels, in such a manner as to obtain a desirable behavior of the extent:

\begin{theorem}\label{tunnel-composition-thm}
Let $(\A,\Lip_\A)$, $(\B,\Lip_\B)$ and $(\alg{E},\Lip_{\alg{E}})$ be three {\Lqcms s}. Let $\tau_1=(\D_1,\Lip_1,\pi_1,\pi_2)$ be a tunnel from $(\A,\Lip_\A)$ to $(\B,\Lip_\B)$ and $\tau_2=(\D_2,\Lip_2,\rho_1,\rho_2)$ be a tunnel from $(\B,\Lip_\B)$ to $(\alg{E},\Lip_{\alg{E}})$. Let $\varepsilon > 0$.

If, for all $(d_1,d_2) \in \sa{\D_1\oplus\D_2}$, we set:
\begin{equation*}
\Lip(d_1,d_2) = \max\left\{ \Lip_1(d_1), \Lip_2(d_2), \frac{1}{\varepsilon}\norm{\pi_2(d_1) - \rho_1(d_2)}{\B} \right\}\text{,}
\end{equation*}
and if $\eta_1 : (d_1,d_2) \mapsto d_1$ and $\eta_2 : (d_1,d_2) \mapsto d_2$, then $(\D_1\oplus\D_2,\Lip)$ is a {\Lqcms} and $\tau_3 = (\D_1\oplus\D_2,\Lip,\pi_1\circ\eta_1,\rho_2\circ\eta_2)$ is a tunnel from $(\A,\Lip_\A)$ to $(\alg{E},\Lip_{\alg{E}})$, whose extent satisfies:
\begin{equation*}
\tunnelextent{\tau_3}{} \leq \tunnelextent{\tau_1}{} + \tunnelextent{\tau_2}{} + \varepsilon \text{.}
\end{equation*}
Moreover, the affine maps $\varphi \in \StateSpace(\D_1)\mapsto \varphi\circ\eta_1$ and $\varphi\in\StateSpace(\D_2)\mapsto\varphi\circ\eta_2$ are isometries from, respectively, $(\StateSpace(\D_1),\Kantorovich{\Lip_1})$ and $(\StateSpace(\D_2),\Kantorovich{\Lip_2})$ into $(\StateSpace(\D_1\oplus\D_2),\Kantorovich{\Lip})$. 
\end{theorem}

\begin{proof}
Let $\alg{F} = \D_1\oplus\D_2$. We first observe that $\Lip$ is a Leibniz lower semi-continuous Lip-norm on $\alg{F}$. To this end, let us show that the map:
\begin{equation*}
\mathrm{N} : d_1,d_2 \in \alg{F} \longmapsto \frac{1}{\varepsilon}\norm{\pi_2(d_1)-\rho_1(d_2)}{\B}
\end{equation*}
is a bridge in the sense of \cite[Definition 5.1]{Rieffel00}, which moreover satisfies the Leibniz inequality. We start with the latter property.

If $d_1,e_1 \in\sa{\D_1}$ and $d_2,e_2 \in \sa{\D_2}$, then:
\begin{equation}\label{tunnel-composition-thm-eq0}
\begin{split}
\mathrm{N}&(d_1e_1,d_2e_2)\\ 
&= \frac{1}{\varepsilon}\norm{\pi_2(d_1e_1) - \rho_1(d_2e_2)}{\B} \\
&\leq \frac{1}{\varepsilon}\norm{\pi_2(d_1)\pi_2(e_1) - \pi_2(d_1)\rho_1(e_2)}{\B} + \frac{1}{\varepsilon}\norm{\pi_2(d_1)\rho_1(e_2) - \rho_1(d_2)\rho_1(e_2)}{\B}\\
&\leq \frac{\norm{d_1}{\D_1}}{\varepsilon}\norm{\pi_2(e_1) - \rho_1(e_2)}{\B} + \frac{\norm{e_2}{\D_2}}{\varepsilon} \norm{\pi_2(d_1) - \rho_1(d_2)}{\D}\\
&\leq \norm{(d_1,d_2)}{\alg{F}}\mathrm{N}(e_1,e_2) + \norm{(e_1,e_2)}{\alg{F}}\mathrm{N}(d_1,d_2)\text{,}
\end{split}
\end{equation}
since $\norm{f_1,f_2}{\alg{F}} = \max\left\{ \norm{f_1}{\D_1},\norm{f_2}{\D_2} \right\}$ for all $(f_1,f_2) \in \alg{F}$. Since $\mathrm{N}$ is a seminorm on $\D_1\oplus\D_2$, it follows from Inequality (\ref{tunnel-composition-thm-eq0}) that $\mathrm{N}$ satisfies:
\begin{equation*}
\mathrm{N}(\Jordan{(d_1,d_2)}{(e_1,e_2)}) \leq \norm{(d_1,d_2)}{\alg{F}}\mathrm{N}(e_1,e_2) + \norm{(e_1,e_2)}{\alg{F}}\mathrm{N}(d_1,d_2)\text{,}
\end{equation*}
and
\begin{equation*}
\mathrm{N}(\Lie{(d_1,d_2)}{(e_1,e_2)}) \leq \norm{(d_1,d_2)}{\alg{F}}\mathrm{N}(e_1,e_2) + \norm{(e_1,e_2)}{\alg{F}}\mathrm{N}(d_1,d_2)\text{,}
\end{equation*}
for all $(d_1,d_2),(e_1,e_2)\in\alg{F}$. From this, it follows easily that $\Lip$ satisfies the Leibniz inequalities (\ref{Leibniz-Jordan-eq}) and (\ref{Leibniz-Lie-eq}).

We now check the three conditions of \cite[Definition 5.1]{Rieffel00}. By construction, $\mathrm{N}$ is continuous for the norm of $\alg{F}$. Moreover, since $\pi_2$ and $\rho_1$ are unital, we have $\mathrm{N}(\unit_{\D_1},\unit_{\D_2}) = 0$. On the other hand $\mathrm{N}(\unit_{\D_1},0) = \frac{1}{\varepsilon} > 0$. 

Last, let $\delta > 0$ and $d_1 \in \sa{\D_1}$. Let $b = \pi_2(d_1) \in \sa{\B}$ and note that since $(\D_1,\Lip_1,\pi_1,\pi_2)$ is a tunnel to $(\B,\Lip_\B)$, we have $\Lip_\B(b) \leq \Lip_1(d_1)$. 

Since $(\D_2,\Lip_2,\rho_1,\rho_2)$ is a tunnel from $(\B,\Lip_\B)$, there exists $d_2 \in \sa{\D_2}$ such that $\rho_1(d_2) = b$ and:
\begin{equation*}
\Lip_2(d_2) \leq \Lip_\B(b) + \delta \leq \Lip_1(d_1) + \delta\text{.}
\end{equation*}
Since $\pi_2(d_1) = b = \rho_1(d_2)$ by construction, we have $\mathrm{N}(d_1,d_2) = 0$. Consequently:
\begin{equation*}
\forall d_1 \in \sa{\D_1}\forall \delta > 0 \quad \exists d_2 \in \sa{\D_2} \quad \max\left\{\Lip_2(d_2), \mathrm{N}(d_1,d_2)\right\} \leq \Lip_1(d_1) + \delta\text{.}
\end{equation*}

The construction is symmetric in $\D_1$ and $\D_2$, so all three conditions of \cite[Definition 5.1]{Rieffel00} are verified.

By \cite[Theorem 5.2]{Rieffel00}, we conclude that $\Lip$ is a lower semi-continuous Lip-norm on $\alg{F} = \D_1\oplus \D_2$, since both $\Lip_1$ and $\Lip_2$ are lower semi-continuous. Thus, $(\alg{F},\Lip)$ is a {\Lqcms}.

Moreover, we observe that the quotient of $\Lip$ on $\sa{\D_1}$ (resp. on $\sa{\D_2}$) is $\Lip_1$ (resp. $\Lip_2$), again by \cite[Theorem 5.2]{Rieffel00}. In particular, this proves that $\varphi \in \StateSpace(\D_1)\mapsto \varphi\circ\eta_1$ and $\varphi\in\StateSpace(\D_2)\mapsto\varphi\circ\eta_2$ are isometries from, respectively, $(\StateSpace(\D_1),\Kantorovich{\Lip_1})$ and $(\StateSpace(\D_2),\Kantorovich{\Lip_2})$ into $(\StateSpace(\D_1\oplus\D_2),\Kantorovich{\Lip})$.

By \cite[Proposition 3.7]{Rieffel00}, since $\Lip_\A$ is the quotient of $\Lip_1$ for $\pi_1$ by definition of $\tau_1$, we conclude that $\Lip_\A$ is the quotient of $\Lip$ for $\pi_1\circ\eta_1$. Similarly, $\Lip_{\alg{E}}$ is the quotient of $\Lip$ for $\rho_2\circ\eta_2$. 

Therefore, we have shown that $\tau_3 = (\alg{F},\Lip,\pi_1\circ\eta_1,\rho_2\circ\eta_2)$ is a tunnel from $(\A,\Lip_\A)$ to $(\alg{E},\Lip_{\alg{E}})$. 

It remains to compute the extent of $\tau_3$.

Let $\varphi \in \StateSpace(\alg{F})$. Since $\alg{F} = \D_1\oplus\D_2$, there exists $t \in [0,1]$, $\varphi_1 \in\StateSpace(\D_1)$ and $\varphi_2\in\StateSpace(\D_2)$ such that $\varphi = t\varphi_1\circ\eta_1 + (1-t)\varphi_2\circ\eta_2$.  

By definition of the extent of $\tau_1$, there exists $\psi \in \StateSpace(\B)$ such that:
\begin{equation*}
\Kantorovich{\Lip}(\varphi_1\circ\eta_1,\psi\circ\pi_2\circ\eta_1) = \Kantorovich{\Lip_1}(\varphi_1, \psi\circ\pi_2) \leq \tunnelextent{\tau_1}{} \text{.}
\end{equation*}
Now, by definition of the extent of $\tau_2$, there exists $\theta \in \StateSpace(\alg{E})$ such that:
\begin{equation*}
\Kantorovich{\Lip}(\psi\circ\rho_1\circ\eta_2, \theta\circ\rho_2\circ\eta_2) = \Kantorovich{\Lip_2}(\psi\circ\rho_1, \theta\circ\rho_2) \leq \tunnelextent{\tau_2}{} \text{.}
\end{equation*}
Let $(d_1,d_2) \in \alg{F}$ with $\Lip(d_1,d_2) \leq 1$. Then by construction of $\Lip$, we have:
\begin{equation*}
\Lip_1(d_1)\leq 1, \Lip_2(d_2)\leq 1\text{ and }\norm{\pi_2(d_1)-\rho_1(d_2)}{\B} \leq \varepsilon\text{.}
\end{equation*}
Thus:
\begin{equation*}
\begin{split}
|\varphi_1(\eta_1(d_1,d_2)) &- \theta(\rho_2\circ\eta_2(d_1,d_2))|\\
&= |\varphi_1(d_1) - \theta(\pi_2(d_2))|\\
&\leq |\varphi_1(d_1) - \psi(\pi_2(d_1))| + |\psi(\pi_2(d_1)) - \psi(\rho_1(d_2))| \\
&\quad + |\psi(\rho_1(d_2)) - \theta(\rho_2(d_2))|\\
&\leq \Kantorovich{\Lip_1}(\varphi_1,\psi\circ\pi_2) + \varepsilon + \Kantorovich{\Lip_2}(\psi\circ\rho_1,\theta\circ\rho_2) \\
&\leq \tunnelextent{\tau_1}{} + \varepsilon + \tunnelextent{\tau_2}{}\text{.}
\end{split}
\end{equation*}
Thus:
\begin{equation*}
\Kantorovich{\Lip}(\varphi_1\circ\eta_1,\theta\circ\rho_2\circ\eta_2) \leq \tunnelextent{\tau_1}{} + \tunnelextent{\tau_2}{} + \varepsilon\text{.}
\end{equation*}

By definition of the extent of $\tau_2$, we can find $\theta_2 \in \StateSpace(\alg{E})$ such that:
\begin{equation*}
\Kantorovich{\Lip}(\varphi_2\circ\eta_2,\theta_2\circ\rho_2\circ\eta_2) = \Kantorovich{\Lip_2}(\varphi_2,\theta_2\circ\rho_2) \leq \tunnelextent{\tau_2}{}\text{.}
\end{equation*}
Since the {\mongekant} $\Kantorovich{\Lip}$ is convex in each of its variable by construction, we conclude:
\begin{equation*}
\begin{split}
\Kantorovich{\Lip}(\varphi, (t\theta+(1-t)\theta_2)\circ\rho_2\circ\eta_2) &\leq \max \left\{ \tunnelextent{\tau_1}{} + \varepsilon + \tunnelextent{\tau_2}{}, \tunnelextent{\tau_2}{} \right\} \\
&= \tunnelextent{\tau_1}{} + \varepsilon + \tunnelextent{\tau_2}{}\text{,}
\end{split}
\end{equation*}
and we note that $t \theta+ (1-t)\theta_2 \in \StateSpace(\alg{E})$. Thus, as $\varphi\in\StateSpace(\alg{F})$ was arbitrary, we conclude:
\begin{equation*}
\Haus{\Kantorovich{\Lip}}(\StateSpace(\alg{F}),(\rho_2\circ\eta_2)^\ast(\StateSpace(\alg{E}))) \leq \tunnelextent{\tau_1}{} + \tunnelextent{\tau_2}{} + \varepsilon \text{.}
\end{equation*}

By symmetry, we would obtain in the same manner that for any $\varphi \in\StateSpace(\alg{F})$ there exists $\theta\in\StateSpace(\A)$ with:
\begin{equation*}
\Kantorovich{\Lip}(\varphi,\theta\circ\pi_1\circ\eta_1) \leq \tunnelextent{\tau_1}{} + \tunnelextent{\tau_2}{} + \varepsilon \text{.}
\end{equation*}

Therefore, by Definition (\ref{tunnel-extent-def}):
\begin{equation*}
\tunnelextent{\tau_3}{} \leq \tunnelextent{\tau_1}{} + \tunnelextent{\tau_2}{} +\varepsilon \text{,}
\end{equation*}
which concludes our proposition.
\end{proof}

Thanks to Theorem (\ref{tunnel-composition-thm}), we can thus define:

\begin{definition}
A pair $(\tau_1,\tau_2)$ of tunnels, with $\tau_1 = (\D_1,\Lip_1,\pi_1,\pi_2)$ and $\tau_2 = (\D_2,\Lip_2,\rho_1,\rho_2)$, such that the co-domain $\codom{\pi_2}$ of $\pi_2$ equals to the codomain of $\rho_1$ is said to be \emph{a composable pair of tunnels}.
\end{definition}

\begin{definition}
For any $\varepsilon > 0$, and for any composable pair $(\tau_1,\tau_2)$ of tunnels, with $\tau_1 = (\D_1,\Lip_1,\pi_1,\pi_2)$ and $\tau_2 = (\D_2,\Lip_2,\rho_1,\rho_2)$, the tunnel:
\begin{equation*}
\left(D_1\oplus\D_2,\Lip,\pi_1\circ\eta_1,\rho_2\circ\eta_2\right)
\end{equation*}
where:
\begin{equation*}
\Lip : (d_1,d_2) \in \D_1\oplus\D_2 \longmapsto \max\left\{ \Lip_1(d_1),\Lip_2(d_2), \frac{1}{\varepsilon}\norm{\pi_2(d_1)-\rho_1(d_2)}{\codom{\pi_2}} \right\}
\end{equation*}
while $\eta_1 : (d_1,d_2)\in\D_1\oplus\D_2\mapsto d_1\in\D_1$ and $\eta_2 : (d_1,d_2)\in\D_2 \mapsto d_2$, as defined in Proposition (\ref{tunnel-composition-thm}), is called the \emph{$\varepsilon$-composition} of $\tau_1$ and $\tau_2$, denoted by $\tau_1\circ_\varepsilon \tau_2$.
\end{definition} 

The dual propinquity is, in fact, a family of metrics indexed by classes of tunnels with enough conditions to ensure that the construction proposed in \cite{Latremoliere13c} leads to an actual metric. The variation we propose in this paper involve an additional constraint compared to \cite[Definition 3.11]{Latremoliere13c}:

\begin{definition}
A class $\mathcal{T}$ of tunnels is \emph{closed under compositions} when, for all $\varepsilon > 0$ and for all composable pair $(\tau_1,\tau_2)$ in $\mathcal{T}$, there exists $\delta \in (0,\varepsilon ]$ such that $\tau_1\circ_\delta \tau_2 \in \mathcal{T}$.
\end{definition} 

A class of tunnels is appropriate with respect to a nonempty class of {\Lqcms s} when it satisfies the following conditions designed to ensure our variant of the dual propinquity is indeed a metric:

\begin{definition}
Let $\mathcal{C}$ be a nonempty class of {\Lqcms s}. A class $\mathcal{T}$ of tunnels is \emph{appropriate} for $\mathcal{C}$ when:
\begin{enumerate}
\item for all $(\A,\Lip_\A)$ and $(\B,\Lip_\B) \in \mathcal{C}$, there exists $\tau \in \mathcal{T}$ from $(\A,\Lip_\A)$ to $(\B,\Lip_\B)$,
\item if there exists an isometric isomorphism $h : (\A,\Lip_\A) \rightarrow (\B,\Lip_\B)$ with $(\A,\Lip_\A)$ and $(\B,\Lip_\B)$ in $\mathcal{C}$, then $(\A,\Lip_\A,\mathrm{id}_\A,h)$ and $(\B,\Lip_\B,h^{-1},\mathrm{id}_\B)$ are elements of $\mathcal{T}$, where $\mathrm{id}_{\alg{E}}$ is the identity automorphism of any C*-algebra $\alg{E}$,
\item if $\tau$ is in $\mathcal{T}$ then $\tau^{-1}$ is in $\mathcal{T}$,
\item if $(\D,\Lip,\pi,\rho) \in \mathcal{T}$, then the codomains of $\pi$ and $\rho$ are elements of $\mathcal{C}$,
\item $\mathcal{T}$ is closed under compositions.
\end{enumerate}
\end{definition}

In particular, we note that if a class of tunnels is appropriate for some nonempty class $\mathcal{C}$ of {\Lqcms s}, then it is compatible with this class \cite[Definition 3.11]{Latremoliere13c}.

We now propose our new metric, which is a natural variation of the dual Gromov-Hausdorff Propinquity, adapted to take advantage of Theorem (\ref{tunnel-composition-thm}):

\begin{definition}\label{propinquity-def}
Let $\mathcal{C}$ be a nonempty class of {\Lqcms s} and let $\mathcal{T}$ be an appropriate class of tunnels for $\mathcal{C}$. The \emph{Gromov-Hausdorff $\mathcal{T}$-Propinquity} $\propinquity{\mathcal{T}}((\A,\Lip_\A),(\B,\Lip_\B))$ between two {\Lqcms s} $(\A,\Lip_\A)$ and $(\B,\Lip_\B)$ is the nonnegative real number:
\begin{equation*}
\inf\left\{\tunnelextent{\tau}{} \middle\vert \tau \in \tunnelset{\A,\Lip_\A}{\B,\Lip_\B}{\mathcal{T}} \right\}\text{.}
\end{equation*}
\end{definition}

We now prove that our propinquity satisfies the triangle inequality.

\begin{theorem}\label{triangle-thm}
Let $\mathcal{C}$ be a nonempty class of {\Lqcms s} and let $\mathcal{T}$ be an appropriate class of tunnels for $\mathcal{C}$. For any {\Lqcms s} $(\A,\Lip_\A)$, $(\B,\Lip_\B)$ and $(\alg{E},\Lip_{\alg{E}})$ in $\mathcal{T}$, we have:
\begin{equation*}
\propinquity{\mathcal{T}}((\A,\Lip_\A),(\alg{E},\Lip_{\alg{E}}))\leq \propinquity{\mathcal{T}}((\A,\Lip_\A),(\B,\Lip_\B)) + \propinquity{\mathcal{T}}((\B,\Lip_\B),(\alg{E},\Lip_{\alg{E}}))\text{.}
\end{equation*}
\end{theorem}

\begin{proof}
Let $\varepsilon > 0$. By Definition of $\propinquity{\mathcal{T}}$, there exists a tunnel $\tau_1 = (\D_1,\Lip_1,\pi_1,\pi_2)$ from $(\A,\Lip_\A)$ to $(\B,\Lip_\B)$ such that:
\begin{equation*}
\tunnellength{\tau_1}{} \leq \propinquity{\mathcal{T}}((\A,\Lip_\A),(\B,\Lip_\B)) + \frac{1}{3}\varepsilon\text{.}
\end{equation*}
Similarly, there exists a tunnel $\tau_2 = (\D_2,\Lip_2,\rho_2,\rho_3)$ from $(\B,\Lip_\B)$ to $(\alg{E},\Lip_{\alg{E}})$ such that:
\begin{equation*}
\tunnellength{\tau_2}{} \leq \propinquity{\mathcal{T}}((\B,\Lip_\B),(\alg{E},\Lip_{\alg{E}})) + \frac{1}{3}\varepsilon\text{.}
\end{equation*}

Since $\mathcal{T}$ is appropriate, there exists $\delta \in \left(0,\frac{\varepsilon}{3}\right]$ such that $\tau_3 = \tau_1\circ_\delta\tau_2 \in \mathcal{T}$. By Theorem (\ref{tunnel-composition-thm}), we have:
\begin{equation*}
\tunnelextent{\tau_3}{} \leq \tunnelextent{\tau_1}{} + \tunnelextent{\tau_2}{} + \frac{\varepsilon}{3} \text{.}
\end{equation*}

Therefore:
\begin{equation*}
\begin{split}
\propinquity{\mathcal{T}}((\A,\Lip_\A),(\alg{E},\Lip_{\alg{E}})) &\leq \tunnelextent{\tau_3}{}\\
&\leq \tunnelextent{\tau_1}{} + \tunnelextent{\tau_2}{} + \frac{\varepsilon}{3}\\
&\leq \propinquity{\mathcal{T}}((\A,\Lip_\A),(\B,\Lip_\B)) + \propinquity{\mathcal{T}}((B,\Lip_\B),(\alg{E},\Lip_{\alg{E}})) + \varepsilon \text{.}
\end{split}
\end{equation*}
As $\varepsilon > 0$ is arbitrary, our proof is complete.
\end{proof}

We conclude our paper with the statement that the propinquity is a metric, up to isometric isomorphism, which is equivalent to the dual propinquity:

\begin{theorem}
If $\mathcal{C}$ be a nonempty class of {\Lqcms s} and if $\mathcal{T}$ be an appropriate class of tunnels for $\mathcal{C}$, then $\propinquity{\mathcal{T}}$ is a metric on $\mathcal{C}$ up to isometric isomorphism and equivalent to $\dpropinquity{\mathcal{T}}$, i.e. for any $(\A,\Lip_\A)$,$(\B,\Lip_\B)$ and $(\D,\Lip_\D)$ in $\mathcal{C}$:
\begin{enumerate}
\item $\propinquity{\mathcal{T}}((\A,\Lip_\A),(\B,\Lip_\B)) \in [0,\infty)$,
\item $\propinquity{\mathcal{T}}((\A,\Lip_\A),(\B,\Lip_\B)) = \propinquity{\mathcal{T}}((\B,\Lip_\B),(\A,\Lip_\A))$,
\item $\propinquity{\mathcal{T}}((\A,\Lip_\A),(\alg{E},\Lip_{\alg{E}}))\leq \propinquity{\mathcal{T}}((\A,\Lip_\A),(\B,\Lip_\B)) + \propinquity{\mathcal{T}}((\B,\Lip_\B),(\alg{E},\Lip_{\alg{E}}))$,
\item $\propinquity{\mathcal{T}}((\A,\Lip_\A),(\B,\Lip_\B)) = 0$ if and only if there exists an isometric isomorphism $h : (\A,\Lip_\A) \rightarrow (\B,\Lip_\B)$,
\item $\dpropinquity{\mathcal{T}}((\A,\Lip_\A),(\B,\Lip_\B)) \leq \propinquity{\mathcal{T}}((\A,\Lip_\A),(\B,\Lip_\B)) \leq 2\dpropinquity{\mathcal{T}}((\A,\Lip_\A),(\B,\Lip_\B))$.
\end{enumerate}
In particular, if $\mathcal{C}$ is the class of all {\Lqcms s} and $\mathcal{LT}$ is the class of all tunnels between {\Lqcms s}, then $\propinquity{\mathcal{LT}} = \propinquity{}$ is a complete metric such that:
\begin{equation*}
\propinquity{}((\A,\Lip_\A),(\B,\Lip_\B)) \leq 2\max\left\{\diam{\StateSpace(\A)}{\Kantorovich{\Lip_\A}},\diam{\StateSpace(\B)}{\Kantorovich{\Lip_\B}}\right\}\text{.}
\end{equation*}
\end{theorem}

\begin{proof}
Assertion (1) follows from the fact that there exists a tunnel in $\mathcal{T}$ between any two elements of $\mathcal{C}$. Assertion (2) follows from the fact that $\mathcal{T}$ is closed under taking the inverse of tunnels. Assertion (3) is given by Theorem (\ref{triangle-thm}). 

We now note that if $\mathcal{T}$ is appropriate for $\mathcal{C}$ then it is compatible with $\mathcal{C}$ by \cite[Definition 3.11]{Latremoliere13c}. Thus the dual Gro\-mov-\-Haus\-dorff propinquity $\dpropinquity{\mathcal{T}}$ is defined, and is a metric on $\mathcal{C}$ up to isometric isomorphism by \cite[Theorem 4.17]{Latremoliere13c}.

Let $\varepsilon > 0$. By \cite[Definition 3.21]{Latremoliere13c}, there exists a journey:
\begin{equation*}
\Upsilon = \left( \A_j,\Lip_j,\tau_j,\A_{j+1},\Lip_{j+1} : j = 1,\ldots,n \right)
\end{equation*}
from $(\A,\Lip_\A)$ to $(\B,\Lip_\B)$ in $\mathcal{C}$ with $\tau_j \in \mathcal{T}$ for all $j\in\{1,\ldots,n\}$ and such that:
\begin{equation*}
\sum_{j=1}^n \tunnellength{\tau_j}{} \leq \dpropinquity{\mathcal{T}}((\A,\Lip_\A),(\B,\Lip_\B)) + \frac{\varepsilon}{2}\text{.}
\end{equation*}
By Proposition (\ref{length-vs-extent-prop}), for any tunnel $\tau$ in $\mathcal{T}$, we have, for all $j\in\{1,\ldots,n\}$:
\begin{equation*}
\tunnelextent{\tau_j}{} \leq 2\tunnellength{\tau_j}{}
\end{equation*}
from which we conclude, using Theorem (\ref{triangle-thm}):
\begin{equation*}
\begin{split}
\propinquity{\mathcal{T}}((\A,\Lip_\A),(\B,\Lip_\B)) &\leq \sum_{j=1}^n \propinquity{\mathcal{T}}((\A_j,\Lip_j),(\A_{j+1},\Lip_{j+1})) \\
&\leq \sum_{j=1}^n \tunnelextent{\tau_j}{}\\
&\leq 2 \sum_{j=1}^n\tunnellength{\tau_j}{} \\
&\leq 2\dpropinquity{\mathcal{T}}((\A,\Lip_\A),(\B,\Lip_\B)) + \varepsilon \text{.}
\end{split}
\end{equation*}
As $\varepsilon > 0$, we get that:
\begin{equation*}
\propinquity{\mathcal{T}}((\A,\Lip_\A),(\B,\Lip_\B))\leq 2\dpropinquity{\mathcal{T}}((\A,\Lip_\A),(\B,\Lip_\B))\text{.}
\end{equation*}

On the other hand, let $\varepsilon > 0$. There exists a tunnel $\tau$ in $\mathcal{T}$ such that $\tunnelextent{\tau}{} \leq \propinquity{\mathcal{T}}((\A,\Lip_\A),(\B,\Lip_\B)) + \varepsilon$. By \cite[Definition 3.21]{Latremoliere13c}, we then conclude immediately that:
\begin{equation*}
\begin{split}
\dpropinquity{\mathcal{T}}((\A,\Lip_\A),(\B,\Lip_\B))&\leq \tunnellength{\tau}{}\\
&\leq\tunnelextent{\tau}{} \leq  \propinquity{\mathcal{T}}((\A,\Lip_\A),(\B,\Lip_\B)) + \varepsilon\text{.} 
\end{split}
\end{equation*}
This concludes Assertion (5).

Last, if there exists an isometric isomorphism $h : (\A,\Lip_\A) \rightarrow (\B,\Lip_\B)$ between any two elements of $\mathcal{C}$, the bridge $(\A,\Lip_\A,\mathrm{id}_\A,h)$ lies in $\mathcal{T}$ and has extent $0$, as is trivially checked.

The converse for Assertion (4) follows from Assertion (5) and \cite[Theorem 4.16]{Latremoliere13c} (the proof of \cite[Theorem 4.16]{Latremoliere13c} could be applied directly and unchanged, except for the use of journeys of length always $1$).

By \cite[Theorem 6.27]{Latremoliere13c} and Assertion (5), we conclude that our propinquity $\propinquity{}$ is complete. The bound provided is a corollary of Assertion (5) and \cite[Corollary 5.6]{Latremoliere13c}.
\end{proof}

Since our propinquity is equivalent to the dual propinquity, all convergence results and comparison to other noncommutative extensions of the Gro\-mov-\-Haus\-dorff distance applies to it. We refer the reader to \cite{Latremoliere13c} for these matters.


\vfill
\end{document}